\documentclass[12pt,letterpaper,twoside]{article}

\usepackage{amsmath,amssymb}
\usepackage{amsthm}
\usepackage{mathtools}
\usepackage{stmaryrd}
\usepackage{fancyhdr}
\usepackage{times}
\usepackage{mathrsfs} 
\usepackage{titlesec}
\usepackage{enumitem}
\usepackage[pdftex,dvips]{geometry}
\usepackage[bottom,multiple,hang]{footmisc}

\input diagxy.tex

\mathtoolsset{centercolon,mathic}

\titlelabel{\thetitle.\ }
\titleformat*{\section}{\normalsize\bfseries\centering}
\titleformat*{\subsection}{\normalsize\itshape}

\newtheoremstyle{fact}
     {\topsep}
     {\topsep}
     {\slshape}
     {}
     {\bfseries}
     {}
     { }
     {\thmname{#1}\thmnumber{ #2.}\thmnote{ \rm (#3)}}

\newtheoremstyle{mylabel} 
        {\topsep}
        {\topsep}
        {\itshape}
        {}
        {\bfseries}
        {}
        { }
        {\thmname{#1}\thmnote{ #3}.} 

\newtheorem{theorem}{Theorem}[section]
\newtheorem{Ltheorem}{Theorem}

\newtheorem*{theorem*}{Theorem} 
\newtheorem{lemma}[theorem]{Lemma}

\newtheorem{corollary}[theorem]{Corollary}

\theoremstyle{definition}

\newtheorem*{remark*}{Remark}

\newtheorem*{question*}{Question}
\newtheorem*{examples*}{Examples}  
\newtheorem{example}[theorem]{Example}
\newtheorem*{example*}{Example}

\theoremstyle{mylabel}
\newtheorem*{Ltheorem*}{Theorem}

\theoremstyle{fact}

\newtheorem{ftheorem}[theorem]{Theorem}

\setenumerate{topsep=0pt,labelwidth=0pt,align=left,labelindent=0pt,
labelsep=0.0in,labelwidth=0.275in,leftmargin=0.275in,label=\rm(\alph*),
parsep=0pt,itemsep=3pt}

\def\proofont{\fontseries{bx}\fontshape{sc}\selectfont}
\def\proofname{Proof. }

\newcommand{\colim}{\operatorname*{colim}}

\makeatletter
\renewenvironment{proof}[1][\proofname]{\par
  \normalfont
  \topsep6\p@\@plus6\p@ \trivlist
  \item[\hskip\labelsep\noindent\proofont #1]\ignorespaces
}{%
  \qed\endtrivlist
}
\makeatother

\makeatletter
\def\@fnsymbol#1{\ifcase#1\or * \or 1 \or 2  \else\@ctrerr\fi\relax}
\let\mytitle\@title

\author{Rafael Dahmen and G\'abor Luk\'acs}

\title{Long colimits of topological groups II: Free groups and  vector spaces\thanks{2010 Mathematics Subject Classification: 
Primary 46A16 46M40; Secondary 46A99 54D50 54D55}}

\begin{document}

\makeatletter
\let\mytitle\@title
\chead{\small\itshape R. Dahmen and G. Luk\'acs / Long colimits of topological groups II: Free groups and  vector spaces}
\fancyhead[RO,LE]{\small \thepage}
\makeatother

\maketitle

\def\thanks#1{}

\thispagestyle{empty}

\begin{abstract}
Topological properties of the free topological group and the free abelian topological group on a space have been thoroughly studied since the 1940s. In this paper, we study
the free topological $\mathbb{R}$-vector space $V(X)$ on $X$. We show that $V(X)$ is
a quotient of the free abelian topological group on $[-1,1]\times X$, and use this to prove topological vector space
analogues of existing results for free topological groups on pseudocompact spaces. As
an application, we show that certain families of subspaces of $V(X)$ satisfy the so-called
{\itshape algebraic colimit property} defined in the authors' previous work.
\end{abstract}

\section{Introduction}

For a topological space $X$, the free topological $\mathbb{R}$-vector space $V(X)$ on $X$ is a 
topological vector space over $\mathbb{R}$ with  a continuous map $\iota_X\colon X \rightarrow V(X)$
that satisfies the universal property that for every continuous map $f\colon X \rightarrow W$ into a topological
$\mathbb{R}$-vector space  $W$, there is a unique continuous linear map  $\bar f \colon V(X) \rightarrow W$
such that $f=\bar f \circ \iota_X$:
\begin{align}
\bfig
\qtriangle/->`->`-->/[X`V(X)`W;\iota_X`f`\exists !\bar f]
\efig 
\end{align}
The free topological group $F(X)$ and the free abelian topological group $A(X)$ on $X$ are defined 
analogously, using similar universal properties. Since the 1940s, the topological properties of $F(X)$ and $A(X)$ have been thoroughly studied (see \cite{TkaTop2} for a survey up to the year 2000).

In this paper, we prove topological vector space analogues of existing results for free topological
groups and abelian topological groups. Our results are based on expressing subspaces of $V(X)$
as quotients of subgroups of $A([-1,1] \times X)$. 

For a topological space $Y$ and its subspace $X$, let $V(X,Y)$ denote the topological vector subspace
of $V(Y)$ spanned by the image of $X$, and similarly, let $A(X,Y)$ denote the topological subgroup of $A(Y)$ generated by the image of $X$.

\begin{Ltheorem} \label{thm:main:VXY}
Let $X$ be a dense subspace of a Tychonoff space $Y$. Then $V(X,Y)$ is canonically isomorphic as an abelian 
topological group to a quotient of $A([-1,1]\times X,[-1,1]\times Y)$. In particular, $V(X)$ is a~quotient
of $A([-1,1]\times X)$ and of $F([-1,1]\times X)$ as a topological group.
\end{Ltheorem}

We use Theorem~\ref{thm:main:VXY} to prove topological vector space analogues of certain results of 
Tkachenko (\cite[Theorems 4--6]{TkaSCN}).

\begin{Ltheorem} \label{thm:main:pseudocompact}
If $X$ is a  pseudocompact Tychonoff space, then every continuous  map $f\colon V(X) \rightarrow \mathbb{R}$
extends to a continuous map $\widetilde f \colon V(\beta X) \rightarrow \mathbb{R}$.
\end{Ltheorem}

\begin{Ltheorem} \label{thm:main:cbly}
Let $X$ be a sequentially compact Tychonoff space such that $X^n$ is normal for every $n \in \mathbb{N}$.

\begin{enumerate}
    \item 
    If $X$ is a sequential space, then so is $V(X)$.
    
    \item
    If $X^n$ is a $k$-space for every $n \in \mathbb{N}$, then $V(X)$ is a $k$-space.
\end{enumerate}
\end{Ltheorem}

As an application, we show that certain families of subspaces of $V(X)$ satisfy the so-called
{\itshape algebraic colimit property} defined in the authors' previous work (\cite{RDGL1}).
Given a directed family $\{G_\alpha\}_{\alpha \in \mathbb{I}}$ of topological groups with closed embeddings as bonding maps, their union $G=\bigcup\limits_{\alpha\in \mathbb I} G_\alpha$ can be equipped with two topologies: the {\itshape colimit space topology} defined as the finest topology $\mathscr{T}$ making each map $G_\alpha \rightarrow G$ continuous, and the {\itshape colimit group topology}, defined as the finest {\itshape group topology} $\mathscr{A}$ making each map $G_\alpha \rightarrow G$ continuous. The former is always finer than the latter, which begs the question of when the two topologies coincide. 
We say that $\{G_\alpha\}_{\alpha \in \mathbb{I}}$ satisfies the {\em algebraic colimit property} (briefly, {\itshape ACP}) if $\mathscr{T}=\mathscr{A}$, that is, if the colimit of $\{G_\alpha\}_{\alpha \in \mathbb{I}}$ in the category $\mathsf{Top}$ of topological spaces and continuous maps coincides with the colimit in the category $\mathsf{Grp(Top)}$ of topological groups and their continuous homomorphisms.

Recall that a family  $\{X_\alpha\}_{\alpha \in \mathbb{I}}$ of subsets of a space is {\itshape long} if every
countable subfamily has an upper bound, that is, for every countable $J\subseteq \mathbb{I}$, there
is $i_0$ such that $X_j \subseteq X_{i_0}$ for every $j \in J$ (\cite[2.2]{RDGL1}).

\begin{Ltheorem} \label{thm:main:ACP}
Let $X$ be a countably compact sequential Tychonoff  space such that $X^n$ is normal for every
$n \in \mathbb{N}$. Let $\{X_\alpha\}_{\alpha \in \mathbb{I}}$ be a long family of subspaces
of $X$ such that  $X = \bigcup\limits_{\alpha \in\mathbb{I}} X_\alpha$. Then:

\begin{enumerate}
    \item 
    $\{V(X_\alpha,X)\}_{\alpha \in \mathbb{I}}$ satisfies ACP and 
    $\colim\limits_{\alpha \in \mathbb{I}} V(X_\alpha,X) = V(X)$; and
    
    \item
    if each $X_\alpha$ is compact, then $\colim\limits_{\alpha \in \mathbb{I}} V(X_\alpha) = V(X)$.
\end{enumerate}

\end{Ltheorem}

It follows from Theorem~\ref{thm:main:ACP} that $V(\omega_1) = \colim\limits_{\alpha < \omega_1} V([0,\alpha])$.  Theorem~\ref{thm:main:ACP} also has another  interesting corollary. The product $\mathbb{L}_{\geq 0}:=\omega_1 \times [0,1)$ equipped with order topology generated by the lexicographic order is called the {\itshape Closed Long Ray}. It turns out that $X=\mathbb{L}_{\geq 0}$ and $X_\alpha=[(0,0),(\alpha,0)]$ satisfy the conditions
of Theorem~\ref{thm:main:ACP}, and consequently 
$V(\mathbb{L}_{\geq 0}) = \colim\limits_{\alpha < \omega_1} ([(0,0),(\alpha,0)])$ (see Corollary~\ref{apps:cor:LongRay}).

The paper is structured as follows. Theorem~\ref{thm:main:VXY} is proven in \S\ref{sect:VXY},
Theorems~\ref{thm:main:pseudocompact} and~\ref{thm:main:cbly} are proven in \S\ref{sect:pseudocompact},
and the proof of Theorem~\ref{thm:main:ACP} and its applications are presented in \S\ref{sect:applications}.

\section{The free topological vector space is a quotient}

\label{sect:VXY}

In this section, we prove Theorem~\ref{thm:main:VXY} by establishing a more elaborate statement. Let $Y$ be a topological space. The composite 
\begin{align}
[-1,1]\times Y \to<500>^{\operatorname{id}_{[-1,1]} \times \iota_Y} 
[-1,1]\times V(Y) \to<500>^{- \cdot -} V(Y)
\end{align}
of $\operatorname{id}_{[-1,1]} \times \iota_Y$ with the scalar multiplication on $V(Y)$ gives rise
to a continuous group homomorphism
\begin{align}  
A([-1,1]\times Y)  \to<500>^{q}  V(Y).
\end{align}

\begin{Ltheorem*}[\ref{thm:main:VXY}$\empty^\prime$]
Let $X$ be a dense subspace of a Tychonoff space $Y$, let 
$p$ denote the restriction of $q$ to $A([-1,1]\times X,[-1,1]\times Y)$, and 
consider the following commutative diagram:
\begin{align}
\bfig
\square|alra|/->`^{(}->`^{(}->`->/<1200,500>[{A([-1,1]\times X,[-1,1]\times Y)}`V(X,Y)`{A([-1,1]\times Y)}`V(Y);p`dense`dense`q]
\efig
\end{align}
Then $p$ and $q$ are quotient homomorphisms of topological groups.
\end{Ltheorem*}

Before turning to the proof of Theorem~\ref{thm:main:VXY}$\empty^\prime$, we first recall a 
well-known result whose proof is, alas, rarely spelled out.


\begin{lemma} \label{lemma:VXTych}
If $X$ is a Tychonoff space, then so is $V(X)$, and 
$\iota_X\colon X\rightarrow V(X)$ is an embedding.
\end{lemma}

\begin{proof}
Let $\mathscr{C}(X,\mathbb{R})$ denote the space of continuous real-valued functions on $X$, and
define $\operatorname{ev}\colon X \rightarrow \mathbb{R}^{\mathscr{C}(X,\mathbb{R})}$ by
$x\mapsto \operatorname{ev}_x := (f(x))_{f \in \mathscr{C}(X,\mathbb{R})}$. The map
$ \operatorname{ev}$ is continuous, and thus induces a~continuous linear homomorphism
$\widetilde{\operatorname{ev}}\colon V(X)  \rightarrow \mathbb{R}^{\mathscr{C}(X,\mathbb{R})}$.
Since $X$ is Tychonoff,   $\operatorname{ev}$ is an embedding, and so $\iota_X$ is an embedding 
as well. Furthermore, $\mathscr{C}(X,\mathbb{R})$ separates the 
points of $X$, and thus  the set $\{\operatorname{ev}_x\}_{x\in X}$ is linearly independent. 
Therefore, $\widetilde{\operatorname{ev}}$ is injective. Hence, $V(X)$ is Hausdorff. Since $V(X)$ is
also a topological group, it follows that it is also Tychonoff.
\end{proof}

\begin{proof}[Proof of Theorem~\ref{thm:main:VXY}$\empty^\prime$.]
It is clear that $p$ and $q$ are surjective.
It suffices to show that $p$ is a quotient, because one may take $X=Y$. 
Put $N:=\ker p$, $E:=A([-1,1]\times X,[-1,1]\times Y)/N$, and let $\bar p\colon E \rightarrow V(X,Y)$ denote the induced continuous bijective homomorphism. 
We prove the statement in two steps, by first showing that $E$ is a topological $\mathbb{R}$-vector space, 
and then proving that $\bar p$ is a homeomorphism.

{\itshape Step 1.} We show that $E$ is a topological $\mathbb{R}$-vector space. 
Let  $\pi\colon A([-1,1]\times X,[-1,1]\times Y) \rightarrow E$ denote the canonical 
projection. The group $E$ is Hausdorff, because $V(X,Y)$ is so (Lemma~\ref{lemma:VXTych}), and so it
admits a completion $\widetilde E$ as a topological group (\cite[1.46]{GLCLTG}). Since
$X$ is dense in $Y$, the subgroup $A([-1,1]\times X,[-1,1]\times Y)$ is dense in $A([-1,1]\times Y)$,
and thus $\pi$ extends to a continuous homomorphism 
$\widetilde{\pi}\colon A([-1,1]\times Y) \rightarrow \widetilde E$  (see \cite[1.45]{GLCLTG}).
Let $j\colon [-1,1]\times Y \rightarrow A([-1,1]\times Y)$ denote the canonical embedding.
\begin{align}
\label{diag:jpi}
\bfig
\square(0,0)|alra|/->`^{(}->``->/<1200,500>[{[-1,1]\times X}`{A([-1,1]\times X,[-1,1]\times Y)}`{[-1,1]\times Y}`{A([-1,1]\times Y)};j_{|[-1,1]\times X}```j]
\square(1200,0)|alra|/->`^{(}->`^{(}->`->/<1200,500>[{A([-1,1]\times X,[-1,1]\times Y)}`E`{A([-1,1]\times Y)}`\widetilde{E};\pi```\widetilde{\pi}]
\efig
\end{align}
For each $n \in\mathbb{N}$, let
\begin{align}
    g_n\colon (-2^n,2^n) \times [-1,1] \times Y & \longrightarrow \widetilde E \\
    (s,t,y) & \longmapsto 2^n \widetilde{\pi} j (\tfrac {st}{2^n},y).
\end{align}
Each $g_n$ is continuous. We claim that the $g_n$ are coherent, that is,
$g_{{n+1}|(-2^n,2^n) \times [-1,1] \times Y}  =g_n$. It suffices
to prove that $g_{n+1}(s,t,x)=g_n(s,t,x)$ for $(s,t,x) \in (-2^n,2^n) \times [-1,1] \times X$,
because $X$ is dense in $Y$. Since $j (\tfrac {st}{2^{n+1}},x) \in A([-1,1]\times X,[-1,1]\times Y)$,
we have
\begin{align}
     g_{n+1}(s,t,x) - g_{n}(s,t,x) & = 
     2^{n+1} \widetilde{\pi} j (\tfrac {st}{2^{n+1}},x) -  2^{n} \widetilde{\pi} j (\tfrac {st}{2^{n}},x) \\
    & = 2^{n} \pi (2 j (\tfrac {st}{2^{n+1}},x) -  j (\tfrac {st}{2^{n}},x))=0,
\end{align}
because $2 j (\tfrac {st}{2^{n+1}},x) -  j (\tfrac {st}{2^{n}},x) \in \ker p = N$.
The coherent continuous maps $\{g_n\}_{n\in\mathbb{N}}$ give rise to a continuous map
$g\colon \mathbb{R} \times [-1,1]\times Y \rightarrow \widetilde E$. By a similar argument,
it follows that $g$ is additive in the first component, that is,
\begin{align}
    g(s_1+s_2,t,y)  =  g(s_1,t,y) + g(s_2,t,y).
\end{align}
Since $\mathbb{R}$ is locally compact, it is exponentiable (\cite[3.4.8]{Engel}). Thus, $g$ corresponds to a continuous map $h\colon [-1,1]\times Y \rightarrow \mathscr{H}(\mathbb{R},\widetilde E)$, where
$\mathscr{H}(\mathbb{R},\widetilde E)$ is the group of all continuous homomorphisms 
$\mathbb{R}\rightarrow \widetilde E$, equipped with the compact-open topology.
By the universal property of free abelian topological groups, $h$ corresponds to a continuous group homomorphism 
$\bar h \colon A([-1,1]\times Y) \rightarrow \mathscr{H}(\mathbb{R},\widetilde E)$. The image of
$A([-1,1]\times X,[-1,1]\times Y)$ under $\bar h$ is contained in $\mathscr{H}(\mathbb{R},E)$, and
so one obtains a continuous group homomorphism
\begin{align}
    \mu \colon A([-1,1]\times X,[-1,1]\times Y) \rightarrow  \mathscr{H}(\mathbb{R},E).
\end{align}
Let $a \in N$. Then $\mu(a)(1) = \pi(a) = 0$. Since $E$ is torsion free (being algebraically isomorphic
to $V(X,Y)$), it follows that $\mu(a)(q)=0$ for every $q\in \mathbb{Q}$, and by continuity,
$\mu(a)=0$. Consequently, $N \subseteq \ker \mu$, and so $\mu$ factors  through a continuous group
homomorphism 
$\bar\mu \colon E \rightarrow  \mathscr{H}(\mathbb{R},E)$.
Using the exponentiability of $\mathbb{R}$ once again, we obtain a continuous 
$\mathbb{Z}$-bilinear\footnote{For greater clarity, {\itshape $\mathbb{Z}$-bilinear} means $m(r_1 + r_2,e) = m(r_1,e) + m(r_2,e)$ and 
$m(r,e_1+e_2) = m(r,e_1) + m(r,e_2)$.} map
$m\colon \mathbb{R} \times E \rightarrow E$.
This shows that $E$ is a topological $\mathbb{R}$-vector
space.

{\itshape Step 2.} We show that $\bar p$ is a homeomorphism by constructing its inverse. Since $E$ is a topological $\mathbb{R}$-vector
space, so is its group completion $\widetilde E$. By the universal property of $V(Y)$, 
the continuous map $\widetilde \pi j_{|\{1\} \times Y}\colon Y \rightarrow \widetilde E$ (see diagram~\ref{diag:jpi})
factors through
a unique continuous homomorphism of topological vector spaces $\widetilde\varphi\colon V(Y)\rightarrow \widetilde E$:
\begin{align}
\bfig
\square(0,0)|alra|/->`^{(}->``->/<1200,500>[X`{A([-1,1]\times X,[-1,1]\times Y)}`Y`{A([-1,1]\times Y)};j_{|\{1\} \times X}```j_{|\{1\} \times Y}]
\square(1200,0)|alra|/->`^{(}->`^{(}->`->/<1200,500>[{A([-1,1]\times X,[-1,1]\times Y)}`E`{A([-1,1]\times Y)}`\widetilde{E};\pi```\widetilde{\pi}]
\Vtriangle(0,-500)/`->`<-/<1200,500>[Y`\widetilde{E}`V(Y);`\iota_Y`\widetilde{\varphi}]
\efig
\end{align}
Put $\varphi:=\widetilde \varphi_{|V(X,Y)}$. Then one obtains the following diagram:
\begin{align}
\bfig
\square(0,0)|alra|/->`^{(}->``/<1200,500>[X`{A([-1,1]\times X,[-1,1]\times Y)}`Y`{V(X,Y)};j_{|\{1\} \times X}```]
\square(1200,0)|arra|/->`->`^{(}->`->/<1200,500>[{A([-1,1]\times X,[-1,1]\times Y)}`E`{V(X,Y)}`\widetilde{E};\pi`p``\varphi]
\Vtriangle(0,-500)/`->`<-/<1200,500>[Y`\widetilde{E}`V(Y);`\iota_Y`\widetilde{\varphi}]
\morphism(0,500)|b|<1200,-500>[X`V(X,Y);\iota_{Y|X}]
\morphism(2400,500)<-1200,-500>[E`V(X,Y);\bar p]
\morphism(1200,0)/^{(}->/<0,-500>[V(X,Y)`V(Y);]
\efig
\end{align}
In order to prove that $\varphi \bar p(v) = v$ for every $v\in E$ (i.e., that
the middle  triangle on the right commutes), it suffices to show that
$\varphi \bar p \pi j(1,x) =   \pi j(1,x)$ for every $x\in X$, because $E$ is spanned by the image of
$\pi j_{| \{1\} \times X }$ as a vector space.
\begin{align}
\varphi \bar p \pi j(1,x) & = \varphi p  j(1,x) \\
& = \varphi \iota_{Y\mid X}(x) \\
& = \widetilde \varphi \iota_Y(x) \\
& = \widetilde \pi j(1,x) \\
& = \pi j(1,x).
\end{align}
This completes the proof.
\end{proof}

Since $A([-1,1]\times X)$ is a topological quotient group of $F([-1,1]\times X)$, 
Theorem~\ref{thm:main:VXY} follows.

\section{Free topological vector spaces on pseudocompact spaces}

\label{sect:pseudocompact}

In this section, we prove Theorems~\ref{thm:main:pseudocompact} and~\ref{thm:main:cbly}. We first recall
two analogous results for free topological groups by Tkachenko.

\begin{ftheorem}[{\cite[Theorem 6]{TkaSCN}}] \label{fact:Tkachenko:pseudocompact}
If $X$ is a pseudocompact Tychonoff space, then every continuous map
$f\colon F(X)\rightarrow \mathbb{R}$ extends to a continuous map
$\widetilde f \colon F(\beta X) \rightarrow \mathbb{R}$.
\end{ftheorem}

\begin{ftheorem}[{\cite[Theorems 4--5]{TkaSCN}}] \label{fact:Tkachenko:cblycompact}
Let $X$ be a countably compact Tychonoff space such that $X^n$ is normal for every $n \in \mathbb{N}$.

\begin{enumerate}
    \item 
    If $X$ is a sequential space, then so is $F(X)$.
    
    \item
    If $X^n$ is a $k$-space for every $n \in \mathbb{N}$, then $F(X)$ is a $k$-space.
\end{enumerate}
\end{ftheorem}

Combining Theorems~\ref{fact:Tkachenko:pseudocompact} and~\ref{thm:main:VXY} yields the following result, which contains Theorem~\ref{thm:main:pseudocompact}.

\begin{Ltheorem*}[\ref{thm:main:pseudocompact}$\empty^\prime$]
If $X$ is a pseudocompact Tychonoff space, then:

\begin{enumerate}
    \item 
    the natural continuous homomorphism $V(X) \rightarrow V(X,\beta X)$ is a topological
    isomorphism; and
    
    \item
    every continuous  map $f\colon V(X) \rightarrow \mathbb{R}$
extends to a continuous map $\widetilde f \colon V(\beta X) \rightarrow \mathbb{R}$.
\end{enumerate}

\end{Ltheorem*}

\begin{proof}
(a) Consider the following commutative diagram, with $i$ and $j$  the natural continuous
bijective homomorphisms: 
\begin{align}
\bfig
\square|alra|/->`->`->`->/<1200,500>[{A([-1,1]\times X)}`V(X)`{A([-1,1]\times X,[-1,1]\times \beta X)}`V(X,\beta X);q`i`j`p]
\efig
\end{align}   
By Theorem~\ref{thm:main:VXY}$\empty^\prime$  applied to the pairs $(X,X)$ and $(X,\beta X)$, 
the maps
$q$ and $p$ are quotients. Since $X$ is pseudocompact, so is $[-1,1]\times X$
(\cite[3.10.27]{Engel}), and thus by Glicksberg's Theorem (\cite[Theorem~1]{Glick2}),
\begin{align}
\beta([-1,1]\times X) \cong [-1,1]\times \beta X.
\end{align}
It is well known that if $Z$ is pseudocompact, then the natural map
$A(Z)\rightarrow A(Z,\beta Z)$ is a topological isomorphism (\cite[2.6.1]{TkaTop2}).
Therefore, 
\begin{align}
    A([-1,1]\times X) 
    \to^i 
    A([-1,1]\times X,\beta([-1,1]\times X))
    \cong 
    A([-1,1]\times X,[-1,1]\times \beta X)  
\end{align}
is  a topological isomorphism. Hence, $j$ is a topological isomorphism too.

(b) By Theorem~\ref{thm:main:VXY}$\empty^\prime$, $V(X)$ is a quotient
of $A([-1,1]\times X)$, which in turn is a quotient of $F([-1,1]\times X)$.
Similarly, $V(\beta X)$ is a quotient of $F([-1,1]\times \beta X)$. Thus,
by part (a), one obtains the following commutative diagrams
with $\pi_i$ being quotients:
\begin{align}
    \bfig
    \square|alra|/->`^{(}->`^{(}->`->/<1200,500>[{F([-1,1]\times X)}`V(X)`{F([-1,1]\times \beta X)}`V(\beta X);\pi_1`dense`dense`\pi_2]
    \efig
\end{align}
Let $f\colon V(X) \rightarrow \mathbb{R}$ be a continuous map.
Then the composite $f\pi_1 \colon F([-1,1]\times X) \rightarrow \mathbb{R}$ is a continuous
map, and by Theorem~\ref{fact:Tkachenko:pseudocompact}, it extends
to a continuous map $f^\prime\colon F([-1,1]\times \beta X) \rightarrow \mathbb{R}$.
Since $\pi_2$ is a quotient map, it suffices to show that $f^\prime$ is
constant on the cosets of $\ker \pi_2$. To that end,
let $x,y \in F([-1,1]\times \beta X)$ be such that
$y=xz$, where $z \in \ker \pi_2$.
Since $F([-1,1]\times X)$ is a dense subgroup of $F([-1,1]\times \beta X)$ and, by part (a),
the natural continuous homomorphism $V(X)\rightarrow V(\beta X)$ is
an embedding, $\ker \pi_1 = F([-1,1]\times X) \cap \ker \pi_2$ is dense in $\ker \pi_2$ (\cite[1.17]{GLCLTG}). Thus, there is a net $\{x_\alpha\}$
in $F([-1,1]\times X)$ such that $x_\alpha \rightarrow x$ and
there is a net $\{z_\alpha\}$ in $\ker \pi_1$ such that 
$z_\alpha \rightarrow z$. Therefore,
\begin{align}
    f^\prime(y) & = \lim f^\prime (x_\alpha z_\alpha) \\
    & = \lim f(\pi_1(x_\alpha z_\alpha)) \\
    & = \lim f(\pi_1(x_\alpha)) \\
    & =\lim f^\prime(x_\alpha) = f^\prime(x).
\end{align}
Hence, $f^\prime$ factors through a continuous map 
$\widetilde f \colon V(\beta X) \rightarrow \mathbb{R}$
that extends $f$.
\end{proof}

Combining Theorems~\ref{fact:Tkachenko:cblycompact} and~\ref{thm:main:VXY} yields the following result.

\begin{Ltheorem*}[\ref{thm:main:cbly}]
Let $X$ be a sequentially compact Tychonoff space such that $X^n$ is normal for every $n \in \mathbb{N}$.

\begin{enumerate}
    \item 
    If $X$ is a sequential space, then so is $V(X)$.

    \item
    If $X^n$ is a $k$-space for every $n \in \mathbb{N}$, then $V(X)$ is a $k$-space.
\end{enumerate}
\end{Ltheorem*}

For sequential spaces, countable compactness and sequential compactness are equivalent  (\cite[3.10.31]{Engel}). 
Consequently, in part (a) of Theorem~\ref{thm:main:cbly}, the condition that $X$ be sequentially compact is not stricter than the conditions in Theorem~\ref{fact:Tkachenko:cblycompact}(a).

\begin{proof}[Proof of Theorem~\ref{thm:main:cbly}.] 
We show that $[-1,1]\times X$ satisfies the conditions of Theorem~\ref{fact:Tkachenko:cblycompact}.
Since $X$ is sequentially compact, so is $[-1,1]\times X$ (\cite[3.10.35]{Engel}). 
In particular, $[-1,1]\times X$ is countably compact.
Let $n \in \mathbb{N}$. Since $X$ is sequentially compact, so is $X^n$ for every $n \in \mathbb{N}$ (\cite[3.10.35]{Engel}). 
Thus, $X^n$ is countably compact, and in particular, it is countably paracompact. 
This implies that
\begin{align}
([-1,1]\times X)^n \cong [-1,1]^n \times X^n
\end{align}
is normal, because $X^n$ is normal (\cite[5.2.7]{Engel}).

(a) The space $[-1,1] \times X$ is sequential, because $X$ is so and $[-1,1]$ is compact (\cite[3.3.J]{Engel}). 
Thus, by Theorem~\ref{fact:Tkachenko:cblycompact}(a), $F([-1,1]\times X)$ is sequential.
Therefore,  by Theorem~\ref{thm:main:VXY}, $V(X)$ is sequential, being a quotient
of $A([-1,1]\times X)$ and consequently a quotient of $F([-1,1]\times X)$.

(b) For every $n \in \mathbb{N}$, the space $([-1,1]\times X)^n \cong [-1,1]^n \times X^n$
is a $k$-space, because $X^n$ is a $k$-space and $[-1,1]^n$ is compact (\cite[3.3.27]{Engel}). 
Thus, by Theorem~\ref{fact:Tkachenko:cblycompact}(b), $F([-1,1]\times X)$ is a $k$-space.
Therefore,  by Theorem~\ref{thm:main:VXY}, $V(X)$ is a $k$-space, being a quotient
of $A([-1,1]\times X)$ and consequently a quotient of $F([-1,1]\times X)$.
\end{proof}

Recall that space $X$ is {\itshape submetrizable} if it admits a continuous injective map into
a metrizable space. For submetrizable spaces, $V(X)$ is sequential if and only if it is a $k$-space
(\cite[3.7]{LinLinLiu}). This, however, does not render Theorem~\ref{thm:main:cbly} redundant.

\begin{example}
There are spaces $X$ that satisfy the conditions of part (b) of Theorem~\ref{thm:main:cbly}
such that $V(X)$ is not sequential. The space $X=[0,\omega_1]$ with the order topology
is compact and sequentially compact, but
is not sequential (because $[0,\omega_1)$ is sequentially closed in $X$). Thus, by
Theorem~\ref{thm:main:cbly}(b), $V([0,\omega_1])$ is a $k$-space, but it is not sequential,
because its closed subspace $\iota_{[0,\omega_1]}([0,\omega_1])$ is not sequential (Lemma~\ref{lemma:VXTych}).
\end{example}

\section{Applications}

\label{sect:applications}

In this section, we present the proof of Theorem~\ref{thm:main:ACP} and two applications.

\begin{Ltheorem*}[\ref{thm:main:ACP}]
Let $X$ be a countably compact sequential Tychonoff  space such that $X^n$ is normal for every
$n \in \mathbb{N}$. Let $\{X_\alpha\}_{\alpha \in \mathbb{I}}$ be a long family of subspaces
of $X$  such that  $X = \bigcup\limits_{\alpha \in\mathbb{I}} X_\alpha$. Then:

\begin{enumerate}
    \item 
    $\{V(X_\alpha,X)\}_{\alpha \in \mathbb{I}}$ satisfies ACP and 
    $\colim\limits_{\alpha \in \mathbb{I}} V(X_\alpha,X) = V(X)$; and
    
    \item
    if each $X_\alpha$ is compact, then $\colim\limits_{\alpha \in \mathbb{I}} V(X_\alpha) = V(X)$.
\end{enumerate}

\end{Ltheorem*}

\begin{proof}
(a) As noted earlier, countable compactness and sequential compactness 
are equivalent for sequential spaces (\cite[3.10.31]{Engel}). 
Thus, by Theorem~\ref{thm:main:cbly}(a), $V(X)$ is sequential, and in particular,
it is countably tight (\cite[1.7.13(c)]{Engel}). By a special case of 
\cite[2.3]{RDGL1}, the statement follows.

(b) Since $X_\alpha$ is compact, the natural continuous homomorphism 
$V(X_\alpha) \rightarrow  V(X_\alpha,X)$ is a topological isomorphism (\cite[3.12]{GabMor}). 
Hence, the statement follows by (a).
\end{proof}

Since $\omega_1$  is countably compact, sequential, and Tychonoff, and $\omega_1^n$ is normal
for every $n \in \mathbb{N}$ (\cite[Corollary]{Conover}), Theorem~\ref{thm:main:ACP} has the following immediate application.

\begin{corollary} \label{apps:cor:omega}
$V(\omega_1) = \colim\limits_{\alpha < \omega_1} V([0,\alpha])$. \qed
\end{corollary}

The space $\omega_1$ is not the only space known to satisfy the conditions of Theorem~\ref{thm:main:ACP}.
The Closed Long Ray, $\mathbb{L}_{\geq 0}$, is also countably compact, sequential, and Tychonoff,
and its compact subspaces $X_\alpha=[(0,0),(\alpha,0)]$ also form a long family. It is not immediate to see
that $\mathbb{L}_{\geq 0}^n$ is normal for every $n\in \mathbb{N}$; however, it
follows from the result of Conover (\cite[Theorem~2]{Conover}).  Therefore, we obtain a second
application of Theorem~\ref{thm:main:ACP}.

\begin{corollary} \label{apps:cor:LongRay}
$V(\mathbb{L}_{\geq 0}) = \colim\limits_{\alpha < \omega_1} V([(0,0),(\alpha,0)])$. \qed
\end{corollary}

\section*{Acknowledgments}

We would like to express our heartfelt gratitude to Karl H. Hofmann for introducing us to each other, and to the organizers of the 2015 Summer Conference on Topology and its Applications for providing an environment conducive for this collaboration to form. We wish to thank David Gauld for the valuable correspondence. We are grateful to Karen Kipper for her kind help in proofreading this paper for grammar and punctuation. We are grateful to the anonymous referee for their detailed and helpful suggestions that have contributed to the articulation and clarity of the manuscript.

{\footnotesize

\bibliography{dahmen-lukacs}

}

\begin{samepage}

\bigskip
\noindent
\begin{tabular}{l @{\hspace{1.6cm}} l}
Rafael Dahmen						    & G\'abor Luk\'acs \\
Department of Mathematics				& Department of Mathematics and Statistics\\
Karlsruhe Institute of Technology		& Dalhousie University\\
D-76128 Karlsruhe					    & Halifax, B3H 3J5, Nova Scotia\\
Germany                			        & Canada\\
{\itshape rafael.dahmen@kit.edu}        & {\itshape lukacs@topgroups.ca}
\end{tabular}

\end{samepage}

\end{document}